\documentclass{actams-author}
\usepackage{amssymb}

  \usepackage{graphicx}
	\usepackage{rotating}
	\usepackage{multirow}
\usepackage[labelformat=empty]{caption}

\numberwithin{equation}{section} 

\begin{document}
\ID{E13-xxx} 
 \DATE{Final, 2013-xx-xx} 
 \PageNum{1}
 \Volume{201x}{}{3x}{x} 
 \EditorNote{$^*$Received May x, 201x; revised x x, 201x. } 

\abovedisplayskip 6pt plus 2pt minus 2pt \belowdisplayskip 6pt
plus 2pt minus 2pt
\def\vsp{\vspace{1mm}}
\def\th#1{\vspace{1mm}\noindent{\bf #1}\quad}
\def\proof{\vspace{1mm}\indent{\bf Proof}\quad}
\def\no{\nonumber}

\newenvironment{prof}[1][Proof]{\indent\textbf{#1}\quad }
{\hfill $\Box$\vspace{0.7mm}}
\def\q{\quad} \def\qq{\qquad}
\allowdisplaybreaks[4]

\newcommand{\opP}{{\mathbf \Phi}}
\newcommand{\opr}{{\mathbf r}}
\newcommand{\opM}{{\mathbf M}}
\newcommand{\opR}{{\mathbf R}}
\newcommand{\opG}{{\mathbf G}}
\newcommand{\opF}{{\mathbf F}}
\newcommand{\opL}{{\mathbf L}}
\newcommand{\opLo}{{\mathbf L}_1}
\newcommand{\opB}{\mathbf B}
\newcommand{\opA}{\mathbf A}
\newcommand{\opBo}{{\mathbf B}_1}
\newcommand{\opBt}{{\mathbf B}_2}
\newcommand{\opAo}{{\mathbf A}_1}
\newcommand{\opAt}{{\mathbf A}_2}

\newcommand{\beq}{\begin{equation}}
\newcommand{\eeq}{\end{equation}}



\AuthorMark{Sidorov,  Sidorov \&  Li}                             

\TitleMark{\uppercase{Nonlinear systems: stability, branching and blow-up}}  

\title{\uppercase{Basins of attraction of  nonlinear systems' equilibrium points: stability, branching and blow-up}        
\footnote{$^{\star}$ Corresponding author}}                 
\author{\sl{Nikolai \uppercase{Sidorov}}}    
   {Irkutsk State  University, K.Marx Str. 1, \zipcode{664025}, Irkutsk, Russia\\
E-mail\,$:$ sidorov@math.isu.runnet.ru }

\author{\sl{Denis  \uppercase{Sidorov$^{\star}$}}}    
{
Energy Systems Institute of SB RAS, Lermontov Str. 130, \zipcode{664033}, Irkutsk, Russia\\
 Institute of Solar-Terrestrial Physics of SB  RAS, Lermontov Str. 126a, \zipcode{664033}, Irkutsk, Russia\\
    E-mail\,$:$ dsidorov@isem.irk.ru}  

\author{\sl{Yong \uppercase{Li}}}    
{ Hunan University,  \zipcode{410082}, Changsha, People's Republic of China\\
    E-mail\,$:$ yongli@hnu.edu.cn}  

\maketitle%

\Abstract{The nonlinear dynamical model consisting the system of differential and operator equations is studied. Here differential equation contains a nonlinear operator acting in Banach space,  a nonlinear operator equation with respect to two elements from different Banach spaces.
This system is assumed to enjoy the stationary state (rest points or equilibrium). 
The Cauchy problem with the initial condition with respect to one of the desired functions is formulated. The second function controls  the corresponding nonlinear dynamic process, the initial conditions are not set. The sufficient conditions of
the global classical solution's existence and stabilization at infinity to the rest point are formulated. It is demonstrated that a solution can be constructed
by the method of successive approximations under the suitable sufficient conditions. If the conditions of the main theorem are not satisfied, then several solutions may exist. Some of  solutions can blow-up in a finite time, while others stabilize to a rest point. The special case of  
considered  dynamical models are nonlinear
differential-algebraic equation (DAE)  have successfully modeled  various phenomena in circuit analysis, power systems, chemical process simulations and many other nonlinear processes.
Three examples illustrate the  constructed theory and the main theorem.
Generalization on the non-autonomous dynamical systems concludes the article. 
}      

\Keywords{nonlinear dynamics, stability, rest point,  branching solution, Cauchy problem, DAE, bifurcation, equilibrium,  blow-up.}        

\MRSubClass{45E10; 65R20}      

\section{Introduction}

Let us consider the system
\begin{equation}
\opA\frac{dx}{dt} = \opF(x,u),
\label{eq1}
\end{equation}
\begin{equation}
0 = \opG(x,u).
\label{eq2}
\end{equation}
Here linear operator $\opA : D \subset X \rightarrow E$
has bounded inverse, nonlinear operators $\opF: X \dotplus U \rightarrow E,$
$\opG: X \dotplus U \rightarrow U$
are continuous in the neighborhoods $||x-x_0||_X \leq r_1,$
$||u-u_0||_U \leq r_2$
of real Banach spaces $X, U;$ $E$ is linear real normed space. 
It is assumed that the following operators  decompositions are valid
\beq
\opF(x,u) = \opF(x_0,u_0)+\opA_1(x-x_0)+\opA_2(u-u_0)+\opR(x,u);
\label{eq3}
\eeq
\beq
||\opR(x,u)|| =o(||x-x_0||+||u-u_0||);
\label{eq4}
\eeq
\beq
\opG(x,u) = \opG(x_0,u_0) + \sum\limits_{k=1}^n d^k \bigl( \opG(x_0,u_0); (x-x_0, u-u_0)\bigr) + \opr(x,u);
\label{eq5}
\eeq
\beq
||\opr(x,u)||_U =o\bigl ((||x-x_0||+||u-u_0||)^n\bigr ).
\label{eq6}
\eeq
In decompositions (\ref{eq3}), (\ref{eq5}) there are derivatives and Frechet differentials
calulated in point  $(x_0,u_0)$ as follows
$$\opAo:=\frac{\partial \opF(x,u)}{\partial x}\biggr|_{x=x_0, u=u_0} \in {\mathcal L}(X\rightarrow E);$$
$$\opAt:=\frac{\partial \opF(x,u)}{\partial u}\biggr|_{x=x_0, u=u_0} \in {\mathcal L}(U\rightarrow E);$$
$$\opA_3:=\frac{\partial \opG(x,u)}{\partial x}\biggr|_{x=x_0, u=u_0} \in {\mathcal L}(X\rightarrow U);$$
$$\opA_4:=\frac{\partial \opG(x,u)}{\partial u}\biggr|_{x=x_0, u=u_0} \in {\mathcal L}(U\rightarrow U);$$
$$d^k \left( \opG(x_0,u_0); (x-x_0, u-u_0)\right) = $$ $$= \sum\limits_{i+j=k} C_k^i \frac{\partial^k \opG(x,u)}{\partial x^i \partial u^j} \biggr|_{x=x_0, u=u_0}(x-x_0)^i (u-u_0)^j;$$
$$\frac{\partial^k \opG(x,u)}{\partial x^i \partial u^j}:\, \underbrace{X\dotplus\cdots \dotplus X}_{i\, \text{times}} \dotplus \underbrace{U\dotplus \cdots \dotplus U}_{j\, \text{times}} \rightarrow U.$$
Obviously, the nonlinear operator equation (\ref{eq2}) should enjoy the real  solution
in order for existence of solution of system (\ref{eq1})--(\ref{eq2}).

Therefore, in this work it is assumed that elements $x_0,\, u_0$  are from
real Banach spaces $X$ and $U$ satisfy operator equations $\opF(x,u)=0,$
$\opG(x,u)=0.$ Therefore, $x_0, u_0$ is stationary solution to system 
(\ref{eq1}), (\ref{eq2}) (steady state, equilibrium point or rest point).

In various power engineering problems (here readers may refer e.g. to \cite{vorop, lit00,bel, lit0}), there the special cases of  system
(\ref{eq1})--(\ref{eq2}) were considered, when $X=E={\mathbb R}^n, \, U={\mathbb R}^m$ are finite spaces. 
In works \cite{lit00,bel,lit0} such systems are known as algebraic-differential systems
and considered with initial Cauchy conditions
\beq
x(0)=\Delta,
\label{eq7}
\eeq
 where $\Delta$  is element from neighborhood of equilibrium point $x_0.$  
Solutions $x(t), \, u(t)$
 on semiaxis $[0, +\infty)$  are constructed  such as $$\lim\limits_{t\rightarrow +\infty} (||x(t)||+||u(t)||)=0.$$
Functions $u(t)$ are selected such as solution $x(t), \, u(t)$ 
is stabilized to the rest point $x_0, u_0$ as  $t\rightarrow +\infty.$

Such problem is important for solution of various automatic 
control problem. In classic works of Russian mathematicians  (Anatoliy Lure, Evgenii Barbashin 
\cite{lit16} et al.) 
the fundamentals of the modern theory for automatic control are considered. 
In number of works (see e.g. \cite{lit0,lit00,bel, joh})
in this field the interesting results of numerical analysis of the electric engineering models are considered. The complexity of this problem is demonstrated, which is caused by solutions' bifurcation and blow-up, ref. \cite{lit0, lit00}.
Therefore, it is important to consider the solvability  of the Cauchy problem  (\ref{eq1}), (\ref{eq2}), (\ref{eq7}) as well as numerical methods to attack this challenging problem both from theoretical and applied sides.


This problem is also important for nonlinear dynamic 
systems' mathematical modeling  using   ``input'' -- ``output''  approach \cite{nonlin, lit7}, when
$u$ is ``input'', and   $y=h(x,u)$ is ``output''.
In case of closed loop the output must satisfy the given criteria in the  form of equation $\opG(y)=0.$
(see Fig. 1)

\begin{figure}
	\centering
		\includegraphics[scale=.41]{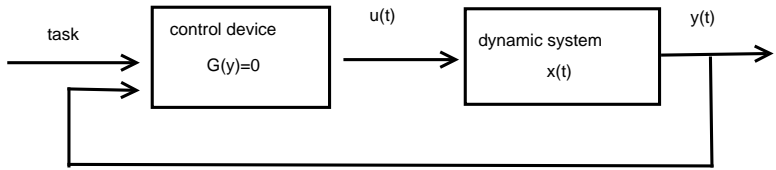}
	\caption{Fig. 1. Scheme of dynamic system with closed loop \cite{bel}.}
	\label{fig:fig1}
\end{figure}

In works \cite{lit00,bel}, \cite{lit0}  only models  (\ref{eq1}), (\ref{eq2}) 
with ordinary differential equations have been  considered. Calculations were performed to demonstrate the 
problem complexity caused by stability analysis and possible blow-up of the solution.

Results of the present work were announced in  \cite{isu1} in Russian language and concentrates on  the nonlinear differential-operator systems in general settings. Constructed theory enables the unified consideration of  ``input-output'' models involving both 
differential and integral-differential equations.

Let  $(0,0)$ be equilibrium point of system (\ref{eq1}), (\ref{eq2})
i.e. in  (\ref{eq3}) -- (\ref{eq6}) and below  $x_0=0,\, u_0=0,$
$\opF(0,0)=0,$ $\opG(0,0)=0.$

\begin{definition}
We call the ball $||x||\leq r$ as
{\it basin of attraction} of equilibrium point $(0,0)$ of system (\ref{eq1}), (\ref{eq2})
such as for arbitrary  $\Delta$ from this ball
there exist solutions $x: [0,+\infty) \rightarrow X,$ $u: [0,+\infty) \rightarrow U$
with initial condition $x(0)=\Delta$
stabilizing to zero on the positive semiaxis.
\end{definition}

\begin{definition}
If basin of attraction of the rest point is nonempty  $(r>0),$  then stationary solution $(0,0)$  of system (\ref{eq1}), (\ref{eq2}) is called {\it asymptotically stable.}
\end{definition}

The {\it objective} of this work is to construct the sufficient condition of non-emptiness of the
basin of attraction of equilibrium points, 
proof of the existence and uniqueness theorem for the solution,
  the construction of sufficient conditions for the nonempty basin of attraction of the equilibrium point, the proof of the existence and uniqueness theorem for the solution
(\ref{eq1}), (\ref{eq2}) with the initial condition $ x(0) = \Delta; $ the development of the method of successive approximations of the solution of the Cauchy problem on semiaxis $ [0, + \infty). $
In addition, for the first time sufficient conditions are formulated for the Cauchy problem's solution branching
 for the system (\ref {eq1}), (\ref {eq2}) with stability analysis of individual branches
of this solution.

The article is organized as follows. The main part of this work (sec.~1-4) concentrated on the case of linear operators $\opA, \opA_4$ have bounded inverse operators. In sec. 1-3 there spectrum of linear bounded operator 
acting from $X$ to $X$
\beq
\opM = \opA^{-1} (\opA_1 - \opA_2 \opA_4^{-1} \opA_3),
\label{eq7o}
\eeq    
will be used.
In sec.~3 it is assumed that $Re\, \lambda \leq -l <0$ for $\lambda \in \sigma (\opM).$
The existence and uniqueness Theorem 1 on the semiaxis $[0,\infty)$ is proved, 
 the asymptotic stability of the Cauchy problem for sufficiently small norm $||x(0)||.$
In sec.~4 the scheme of successive approximations of the desired solution is given.
In sec.~5 there classes of systems  (\ref{eq1}), (\ref{eq2}) are considered under
condition of  $\opA_4 =0,$  i.e.  $\frac{\partial G(x,u)}{\partial u}|_{x=u=0} = 0.$
It is shown that in this case branching of the solution of the Cauchy problem can occur when there exist several solutions $ x(t), \, u(t) $ with  condition $ x(0) = \Delta$ in sec. 6.
  Some of the branches extend to the whole semi-axis $ [0, +\infty) $
and they stabilize to zero as $ t \rightarrow + \infty, $
and others can collapse (go to infinity). Illustrative examples are given.
Generalization on the non-autonomous dynamical systems concludes the article in sec. 7.

\section{Reduction of a non-linear system in the neighborhood of an\\ equilibrium point to a single differential equation}


Let  $\opG(x,u) = \opA_3 x +\opA_4 u + \opr(x,u),$
where $||\opr(x,u)|| = o(||x||+||u||).$

\begin{lemma}
Let operator $\opA_4$ has bounded inverse, here $\opA_4 = \frac{\partial \opG(x,u)}{\partial u}|_{x=u=0}$
is Frechet derivative. Then for arbitrary ball
$ S_1 : ||x|| \leq r_1$ exists ball $S_2: ||u||\leq r_2$ such as for any $x \in S_1$
equation (\ref{eq2}) has unique continuous solution $u(x)$ in ball $S_2.$
In this case we have an asymptotic representation of the solution as $||x|| \rightarrow 0:$
\beq
u = -\opA_4^{-1} \opA_3 x + o(||x||),
\label{eq8}
\eeq
where $\opA_3 = \frac{\partial \opG(x,u)}{\partial x}\biggr|_{x=u=0}$ is Frechet derivative 
on $x$.
\end{lemma}
\begin{proof}
Equation (\ref{eq2}) due to invertability of operator  $\opA_4$  and equality $\opG(0,0)=0$
can be reduced to  $u=\opP(u,x).$ Here $\opP(u,x) = - \opA_4^{-1} \opA_3 x - A_4^{-1} r(x,u).$
Hence for $\forall q \in (0,1)$ exists neigborhood  $||u||\leq r_2, $
in which operator  $\opP$ will be uniformly contracting on $x$ from the ball $||x||\leq r_1$: 
$||\opP(u_1,x) - \opP(u_2,x)|| \leq q||u_1-u_2||.$
Moreover, in this neigborhood the following inequality is valid
$||\opP(u,x)||  \leq q r_2 + ||\opP(0, x)||.$
Since $\opP(0,0) =0,$ then due to continuity  $\opP(0, x)$
there exists $r$ in the interval $(0,r_1]$ such as  $||\opP(0,x)|| \leq (1-q)r_2$
for $||x|| \leq r.$ Hereby, for any  $x$ from ball $||x||\leq r $
and $\forall u$ from
 $||u||\leq r_2$
operator $\opP$ will be contracting and transfers ball $S(0,r_2)$ into itself.
Hence, using the principle of contracting mappings, the sequence
$\{ u_n\},$ where $u_n = \opP(u_{n-1} x,)$ $x \in S(0,r),$ $u_0=0$ 
converges uniformly on $x$  to unique solution $u(x)$ in the ball $||u||\leq r_2.$ 
Since  $u_n(x) \sim -\opA_4^{-1} \opA_3 x$ as $||x|| \rightarrow 0,$ 
then we have asymptotics
\beq
u(x) \sim -\opA_4^{-1} \opA_3 x 
\label{op9}
\eeq
as $||x|| \rightarrow 0.$
\end{proof} 
From Lemma 2.1 it follows 
\begin{lemma}
Exists neigborhood  $||x||\leq r,$ $||u||\leq r_2,$ such as system (\ref{eq1}), (\ref{eq2})
can be uniquely reduced to the following  differential equation\beq
\opA \frac{dx}{dt} = f(x).
\label{eq10}
\eeq
Here mapping  $f:\, S(0,r) \subset X \rightarrow E_2 := \opF(x,u(x)) $
is as follows
\beq
f(x) = \opA_1 x - \opA_2 \opA_4^{-1} \opA_3 x +L(x).
\label{eq11}
\eeq
Nonlinear mapping $L: \, X \rightarrow E$ satisfies the estimate  $||L(x)|| = o(||x||).$
\end{lemma}
\begin{proof}
For proof it is sufficient to substitute $u(x),$ defined from (\ref{eq2}), into  (\ref{eq3}) of mapping $\opF(x,u).$
\end{proof}

\begin{remark}
After determination of $x(t)$ one may find the approximate $u(t).$ 
Indeed, let $x(t)$ is solution to differential equation (\ref{eq10}) constructed in Lemma 2.
Then function $u(t)$ using Lemma 1 is constructed by following asymptotic $u(t) \sim -\opA_4^{-1}\opA_3 x(t)$ as   $||x|| \rightarrow 0.$
\end{remark}

\section{An a priori estimate of  the Cauchy problem's solution}

Let us consider the system  (\ref{eq1}), (\ref{eq2}) with initial Cauchy condition
  $x(0)=\Delta.$

\begin{lemma}
Let $x:[0,+\infty) \rightarrow X$ be solution to Cauchy problem for 
system   (\ref{eq1}), (\ref{eq2}), where $||\Delta||$ in the initial condition  is sufficiently small. 
Let $Re \, \lambda \leq -l <0 $ for all $\lambda \in \sigma(\opM).$ Then there are $ C\geq 1$ and $\varepsilon \in (0,l),$ such as $||\exp \opM t||_{\mathcal{L}(X\rightarrow X)} \leq C e^{-lt}$
and $||x(t)||_X \leq C ||\Delta||e^{(\varepsilon -l)t}$
for  $t \in [0,+\infty).$
\end{lemma}

\begin{proof}
 Using Lemma 2.1 let us find function $u(t) $ using successive approximations.
Substitution of this function in the right hand side of eq. (\ref{eq1}) leads to eq. 
(\ref{eq10}). Using identity  $\opA^{-1} f(x) = \opM x +\opA^{-1} L(x)$ 
we get
\beq
\frac{dx}{dt} = \opM x +\opA^{-1}L(x),\,\,\, x(0)=\Delta.
\label{eq12}
\eeq 
 Therefore 
\beq
x(t) = \exp \opM t \Delta +\int_0^t \exp \opM(t-s) \opA^{-1} L(x(s))\, ds =: \opP(x,t)
\label{eq13}
\eeq
 is Volterra integral equation for determination of the desired solution $x(t)$ of Cauchy 
problem. It is known (see e.g., \cite{lit4}), that due to the conditions of  Lemma 3.1 the operator 
exponential on the spectrum $\sigma(\opM)$
satisfies the following estimate $$||\exp \opM t||_{{\mathcal L}(X \rightarrow X)} \leq C e^{-lt},\,\,\, C\geq 1.$$
Therefore, we have the inequality 
$$||x(t)||\leq C e^{-lt} ||\Delta|| +C ||\opA^{-1}|| \int_0^t e^{-l (t-s)} ||L(x(s))||\, ds. $$
For $\forall\varepsilon >0$ due to the estimate $||L(x)|| =o(||x||)$
there exists  $\delta = \delta(\varepsilon)>0$ such as  $||\opA^{-1}|| ||L(x)||\leq \frac{\varepsilon}{C}||x||$ for $||x||\leq \delta.$
Therefore, until  $||x(s)||\leq \delta$ there will be inequality
$e^{lt} ||x(t)|| \leq C||\Delta|| +C \int_0^t e^{ls} \frac{\varepsilon}{C} ||x(s)||\, ds.$
Hence, in view of the Gronwall-Bellman inequality, we have
$$e^{lt} ||x(t)|| \leq C||\Delta||e^{\int_0^t \varepsilon ds}  =C ||\Delta|| e^{\varepsilon t}. $$
 Therefore $||x(t)|| \leq C||\Delta|| e^{(\varepsilon - l)t}$ for $t>0.$
Let us select $\varepsilon\in (0,l).$  Then $||x(t)||\leq C||\Delta|| e^{-(l-\varepsilon)t} < C ||\Delta||$  for $t \in (0,+\infty).$ Let in initial condition $x(0)=\Delta$
 sufficiently small  $\Delta$ is selected, i.e. $||\Delta||<\frac{\delta}{C}.$
 Since $C\geq 1$  then  $||x(0)||<\delta.$  Then $||x(t)||<\delta$ for $\forall t \in [0,+\infty)$
 and condition $||x||\leq \delta$ is satisfied.
Moreover, $\sup\limits_{0\leq t <\infty} ||x(t)||<\delta.$ For $\forall t^* \in [0,\infty)$ 
$||x(t^*)||<\delta.$
Thus, Lemma 3.1 is valid, a continuous solution exists and it is unique on semiaxis $ [0, + \infty). $ 
\end{proof}

A priori assessment of the solution justifies the 
continuation of a local continuous solution of the Cauchy problem to the entire interval $ [0, + \infty). $ In the monograph \cite{lit6} the proof of the possibility of continuous continuation
of the solution was also based on the presence of an a priori estimate of the solution. A priori estimate of the solution, as a rule, is used in proving nonlocal existence theorems
by the method of continuation with respect to a parameter, see, for example, sec. 14 in \cite{lit3}.

 \section {Existence, uniqueness, and asymptotic stability}

\begin{theorem}
Let $(0,0)$ be equilibrium point of  system (\ref{eq1}), (\ref{eq2}).
Let $Re \lambda \leq -l <0$ for all $\lambda \in \sigma (\opM),$
 $||\Delta||$ is sufficiently small. Then system (\ref{eq1}), (\ref{eq2})
with condition $x(0) = \Delta$  enjoys unique solution $x: [0,+\infty) \rightarrow X,$
$u: [0, +\infty) \rightarrow U.$ Moreover, $\lim\limits_{t\rightarrow +\infty}(||x(t)||+||u(t)||)=0.$
\end{theorem}
\begin{proof}

By virtue of Lemmas 1 and 2 and the obvious validity of Picard's theorem for the equation (\ref{eq12}), the Cauchy problem (\ref{eq1}), (\ref{eq2}), $ x(t_0) = x_0 $ has a unique local solution for $ \forall t_0 \in [0, \infty). $ Therefore, the set of values of the arguments $ t, $ for which the local solution continuously extends, is open in any interval
$ [t_0, + \infty). $

Since the solution of the Cauchy problem  for sufficiently small $ || x(0) || $
 on the basis of Lemma 3, satisfies the a priori estimate $ x(t) <\delta \, $ for $ \forall t \in (0, + \infty) $ and does not reach $ \delta,$ then the set of values of the arguments $ t, $ on which the solution can be continued, will be closed. Therefore, on the basis of known facts about the method of continuation with respect to a parameter (see, for example, \cite{lit3}, sec.14) the solution of the equation (\ref{eq12}) continuously extends to the entire interval $ [0, + \infty). $
In view of Lemmas 1, 2 and 3, the desired functions $ x(t), u(t) $ stabilize as $ t \rightarrow + \infty $ to  equilibrium point $ (0,0). $ 
\end{proof}

Additional information on operator $L(x)$  in equation (\ref{eq13})
enables lower estimation of $||x||,$  formulated in Theorem 1.
Indeed, let positive continuous decreazing function  $q(r)$ is found
$q(0)=0$  such as $$||L(x_1)-L(x_2)||\leq q(r) ||x_1-x_2||$$
for any $x_1, x_2$ from ball $||x||\leq r.$
 Then for operator $\opP$  in equation  (\ref{eq13})  we have an estimate
$$\sup\limits_{0\leq t < \infty} ||\opP(x_1,t) - \opP(x_2,t)|| \leq $$ $$\leq C||\opA^{-1}|| \int_0^t e^{-l(t-s)}  ||L(x_1(s)) - L(x_2(s))||\, ds \leq $$  $$\leq \frac{c}{l} ||A^{-1}|| q(r) \sup\limits_{0\leq s < \infty} ||x_1(s)-x_2(s)||.$$
Let us select $r^*>0$ such as for  $\forall r \leq r^*$  the following inequality
 $$\frac{C}{l} ||A^{-1}|| q(r) \leq q^* <1$$ takes place. Then operator $\opP$ 
will be contructing in the ball $S(0,r^*)$ and for  $||x||\leq r^*$ and $\forall t \in [0,+\infty)$
$$||\opP(x,t)||\leq ||\opP(x,t) - \opP(0,t)|| + ||\opP(0,t)|| \leq q^* r^* +$$
$$+||\exp \opM t \Delta|| \leq q^* r^* +C ||\Delta||.$$
 Let us now select $||\Delta|| \leq \frac{(1-q^*)r^*}{C}.$

For such a sufficiently small initial value $ \Delta $, the contraction operator $ \opP $ maps the ball $ S(0, r^*) $
into itself and hence the Cauchy problem (\ref{eq12}) has a unique global solution.

\begin{example}
$$\left\{\begin{array}{l} {\frac{\partial x(t,z)}{\partial t} = -x(t,z) + x^3(t,z) + u^2(t,z)}, \\ {x(0,z) = \Delta(z) },\,\,\,\,\,\,\, t \in [0, +\infty), \,\, z \in [0,1],\\
{u(t,z) + \int_0^1 zs u(t,s)\, ds + x^2(t,z) + u^2(t,z) = 0,}  \end{array}\right. $$
$|\Delta(z)|\leq \varepsilon,$ $\varepsilon$ is sufficiently small. 
Here $\opA = \opA_1 = I, \opA_3=0,$ $X=E=U={\mathcal C}_{[0,1]}.$
Operator $\opA_4 = I + \int_0^1 zs[\cdot]\, ds$
has  bounded inverse  $$\opA_4^{-1} = I - \frac{2}{3} \int_0^1 zs [\cdot]\, ds,$$
$\opM = -I.$ If  $|x(t,z)|$ is sufficiently small then sequence $\{ u_n(t,z) \},$ where
$$u_n(t,z) = -x^2(t,z) - u_{n-1}^2(t,z) +\frac{2}{3}\int_0^1 zs \{x^2(t,s) +u_{n-1}^2(t,s)\}\, ds,$$
$u_0(t,z)=0$ converges and enables the algorithm for construction of solution $u(t,z)$ 
of integral equation as function of  $x(t,z)$.   
Substituting this solution into a differential equation, we reduce the problem to the following differential equation:
 $$\frac{\partial x(t,z)}{\partial t} = - x(t,z) + {\mathcal O}(||x||^2),\,\,x(0,z) = \Delta(z).$$
Here $||u|| = {\mathcal O} (||x||^2).$

Therefore, this model, consisting of differential and integral equations satisfies 
conditions of the Theorem 1 and on the semiaxis $[0,\infty)$ enjoy unique 
continuous solution  $x(t,z),$  $u(t,z),$ stabilizing as  $t\rightarrow +\infty$ to the rest point 
 $(0,0)$ if $\max\limits_{0\leq z \leq 1} |\Delta(z)|$ is sufficiently small. 
\end{example}

\section{On the construction of a solution of a nonlinear system by the method of successive approximations}

Under the conditions of Theorem 1, the desired solution $ x(t), u(t) $ of the system (\ref{eq1}), (\ref{eq2})
with the condition $ x(0) = \Delta $ can be constructed without prior  system's reduction to
one differential equation. Indeed, we introduce two sequences
$ \{x_n (t) \}, \{ u_n(t)\}$ with conditions $x_n(0)=\Delta, \, n=0,1,\dots$,
where $||\Delta||$ is sufficiently small. Let 
$u_0=0,$  and $||\Delta||$ is sufficiently small,
$x_n(t)$ is solution to Cauchy problem 
$\opA \frac{d x_n}{d t} = \opF(x_n, u_{n-1}), \, x_n(0)=\Delta, n=1,2,\dots.$
Obviusly solution $x_n(t)$ exists and unique for $t\geq 0$ due to the Theorem 1.

Next, let us construct functions $u_n$ using the iternations $u_n=u_{n-1}+w_n,$  where $u_0=0.$ 
Due to invertability of the operator $\opA_4,$ functions
$w_n$
can be found from the linear equation
$\opA_4 w_n +\opG(x_n, u_{n-1}) = 0, n=1,2,\dots.$
Then, under Theorem 1 assumptions $\lim\limits_{n\rightarrow \infty} x_n(t) = x(t),$
$\lim\limits_{n\rightarrow \infty} u_n(t) = u(t),$  $\lim\limits_{t\rightarrow \infty} (||x(t)||+||u(t)||)=0.$
It is essential to require small $||\Delta||$ otherwise solution to 
nonlinear differential equation may blow-up in the point 
 $t^*$ (refer to examples below).

\begin{example}
Let us consider the system 
$$\left\{\begin{array}{l} {\frac{dx(t)}{dt}  = -\frac{x(t)}{2} - u(t)+x^2(t)}, \\ 
{ 0 = 2 u(t) -x(t) +2 u(t) \sin u(t) -x(t) \sin u(t).  } \\ \end{array}\right. $$
with initial condition $x(0)=\Delta,$ $0\leq t < +\infty.$
The replacement $u(t) = \frac{x(t)}{2}$ will reduce this system to Cauchy problem
$$\left\{\begin{array}{l} {\dot{x}(t) = -x(t)+x^2(t)}, \\ {x(0)=\Delta}. \\ \end{array}\right. $$
It is easy to verify that the latter model has an exact solution
 $$x(t) = \frac{\Delta}{e^t (1-\Delta)+\Delta}.$$
Let us demonstrate that point $t^* = \ln \frac{\Delta}{\Delta - 1}$
may appear to be blow-up of the constructed solution.
We consider the following 4 cases.

{\it Case 1.}
If $\Delta \in (0,1),$ then blow-up point $t^*$ is complex, solution is continuous for $t \in (0,+\infty)$ and stabilizing to the rest point  $x=0$ as $t \rightarrow +\infty$ (ref. Fig. 2).

\begin{figure}
	\centering
		\includegraphics[scale=.31]{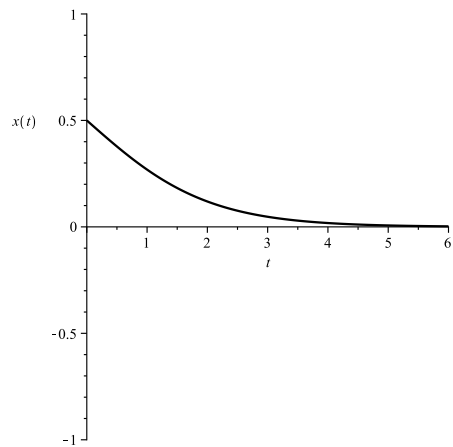}
	\caption{Fig. 2. $\Delta \in (0,1).$}
	\label{fig:fig1}
\end{figure}

{\it Case 2.}
If $\Delta \in (-\infty, 0),$  then blow-up point is negative, and on  semiaxis $[0,+\infty)$  solution is continuous and stabilizing to the rest point  $x=0$  as $t\rightarrow +\infty$ (ref. Fig. 3).

\begin{figure}
	\centering
		\includegraphics[scale=.31]{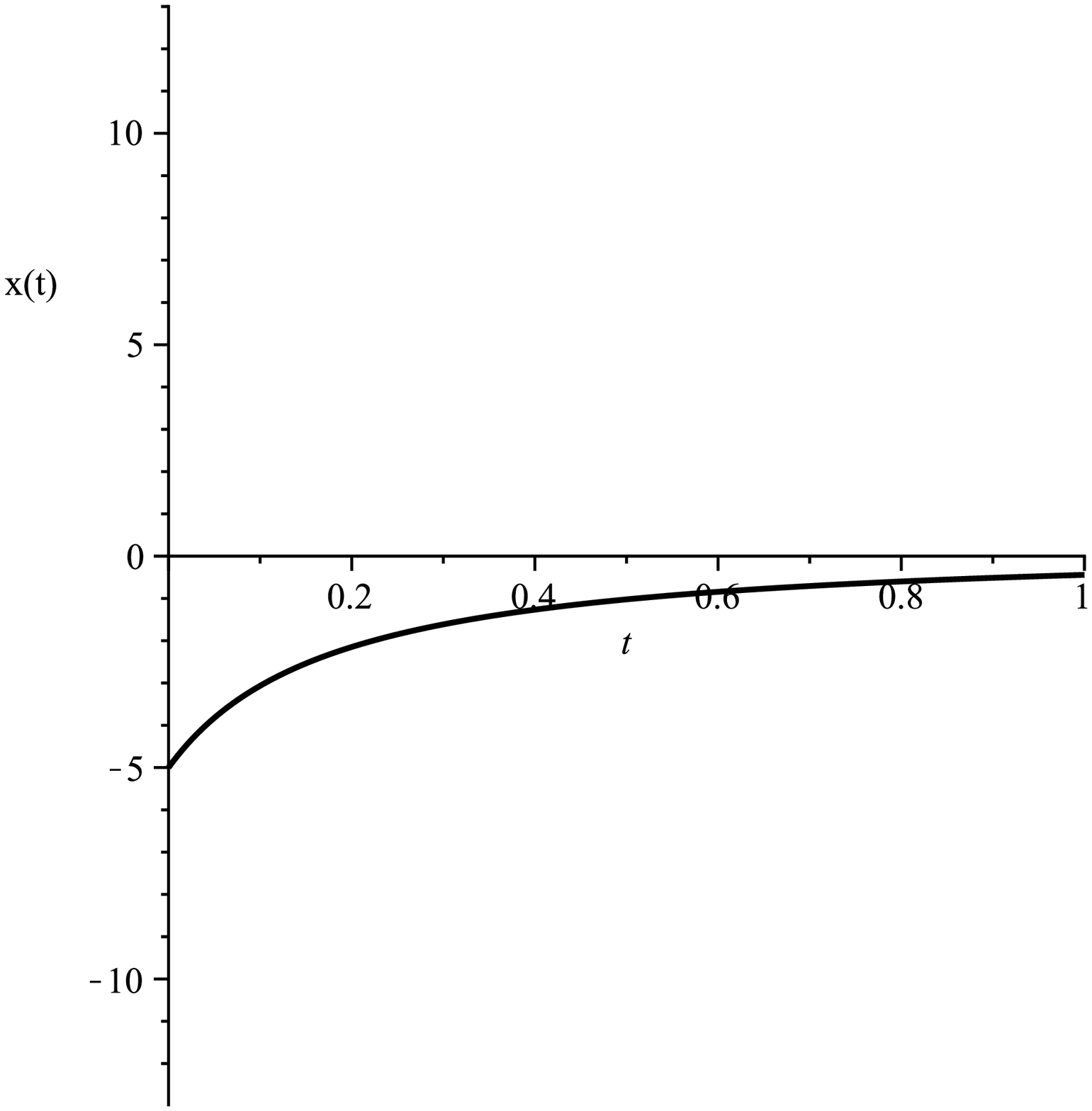}
	\caption{Fig. 3.  $\Delta  < 0.$}
	\label{fig:fig2}
\end{figure}

{\it Case 3.}
If $1< \Delta < \infty,$  then solution  blows-up for $t^* = \ln \frac{\Delta}{\Delta-1},$
where $\frac{\Delta}{\Delta-1}>0.$ If $t>t^*$ then solution is continuous and also 
stabilizing to the rest point $x=0$ as $t\rightarrow +\infty $ (ref. Fig. 4).

{\it Case 4.} For  $\Delta=0$ and $\Delta=1$ we get the stationary solutions.

\begin{figure}
	\centering
		\includegraphics[scale=.32]{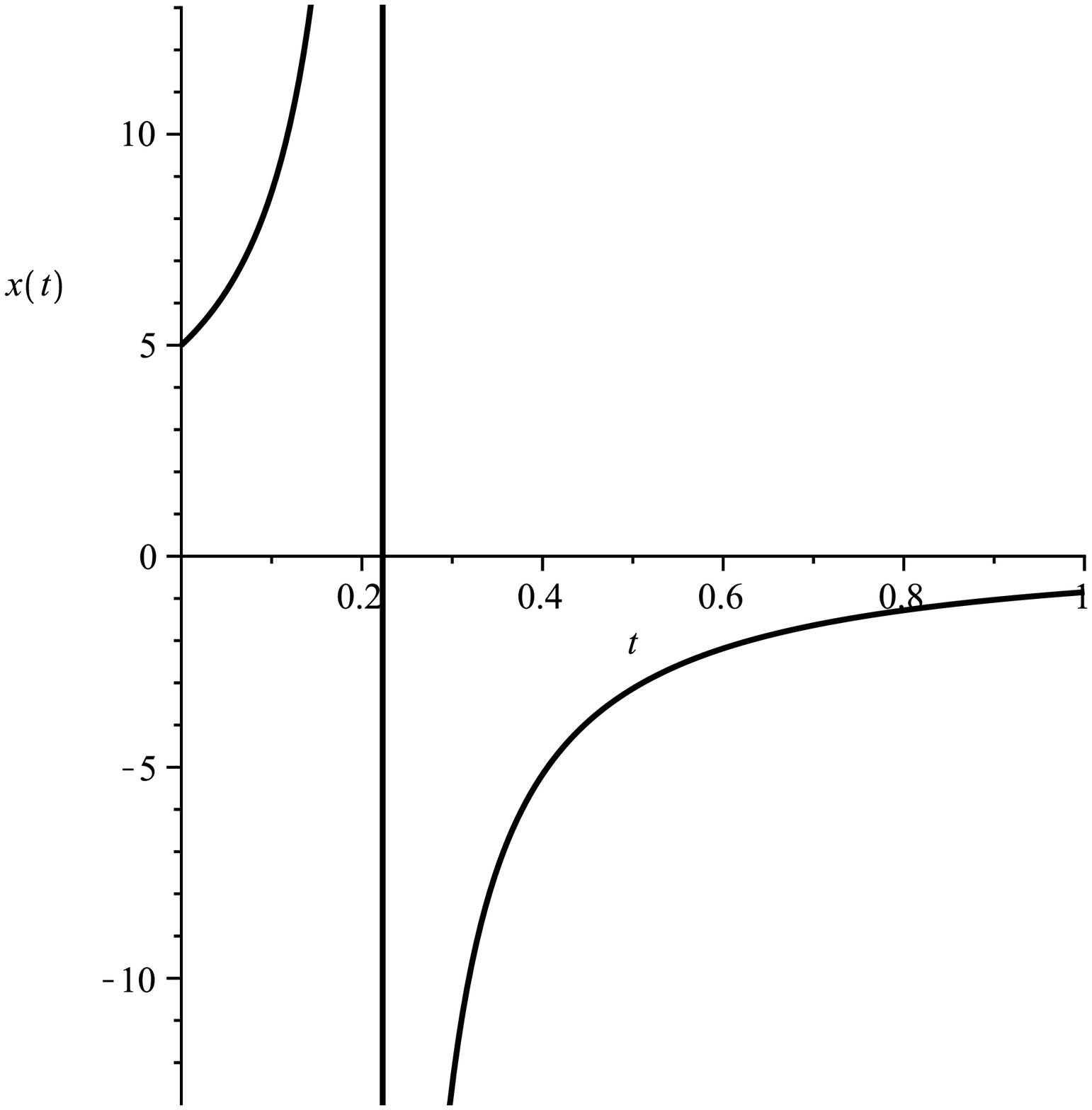}
	\caption{Fig. 4. $\Delta > 1.$}
	\label{fig:fig1}
\end{figure}

\end{example}

{Based on the above, the following conclusion can be drawn.} In example 2 for $\Delta \in (-\infty,1)$ exists the unique solution to the Cauchy problem as $t\geq 0,$ this solution
stabiliziing to the rest point as $t\rightarrow +\infty.$
It is to be mentioned that for  $\Delta >1$  the Cauchy problem's solution will blow-up
on the finite time  $\ln \frac{\Delta}{\Delta -1}.$

\begin{remark}
The absence of real rest points can generate solutions with a countable set of blow-up points.
\end{remark}

%

\section{Branching of the Cauchy problem's solution}

Let  in (\ref{eq1}), (\ref{eq2}) $E=X={\mathbb R}^n,$ $U={\mathbb R}^m,$
$\opA$ is unit $n\times n$ matrix.
The broad class of the systems  (\ref{eq14}), (\ref{eq15}) appear in various applications
\beq
\left\{\begin{array}{l} {\frac{dx_i}{dt} = \sum\limits_{k=1}^n a_{ik} x_k + \sum\limits_{k=1}^m b_{ik} u_k +R_i(x; u),} \\ {x_i(0)=\Delta_i, i=1,\dots, n,} \\
 \end{array}\right. 
\label{eq14}
\eeq

\beq
{0 = q_k(x; u), \, k=1, \dots, m},
\label{eq15}
\eeq
where $$x:=(x_1, \dots, x_n)^T, \,\, u:=(u_1, \dots, u_m)^T, \,\, m\leq n.$$
Let $R_i(x,u) = o(||x|| +||u||),$
$$q_k(x,u)=\opA_k(x_k,u_k)+x_k^{N_k+1} r_k(x,u), \, k=1,\cdots, m, $$
 $$\opA_k(x_k,u_k) := \sum_{s=0}^{N_k} m_{ks} x_k^{N_k-s} u_k^s +o((|x_k|+|u_k|)^{N_k}),$$
$r_k(x,u) = o(1)$ ïðè $||x||+||u||\rightarrow 0,\, N_k\geq 2.$

We will search for small solutions of (\ref{eq15})
$u=u(x) \rightarrow 0$ as $||x||\rightarrow 0$
in the form of products $u_i=x_i w_i(x), i=1,\dots,m$  under condition
 $$\sum\limits_{i=1}^m |w_i(0) |\neq 0.$$
Then functions  $w_i$ must satisfies the following system
$$x_k^{N_k} \sum_{s=0}^{N_k} m_{ks} w_k^s(x_k) + x_k^{N_k+1} r_k(x, x_1 w_1, \dots, x_m w_m) =0,$$ $k=1,\dots,m.$
After reduction, we come to the system
$$\sum_{s=0}^{N_k} m_{ks} w_k^s(x_k) + x_k {r}_k(x,x_1 w_1,\dots , x_m w_m)=0, \, k=1,\dots, m.$$
 Therefore, the vector $w(0)=(w_1(0),\dots,w_m(0))^T$ 
consists of the polynomials roots 
$$\sum_{s=0}^{N_k} m_{ks} w_k^s(0)=0,\, k=1,\dots, m.$$

Let  $w_k^*,$ $k=1,\dots, m$ be simple roots of the corresponding 
polynomials.
On the basis of the implicit function theorem, these roots will have a small solution of
systems (\ref{eq15}) of the form
$$u_k(x) = x_k w_k^* +r_k(x),\, k=1,\dots,m, $$
where $|r_k(x)| = o(||x||).$
Functions $r_k(x)$ for small $x$ can be approximated using successive 
approximations. Substitution of  $u_k(x)$ into differential equations  $(\ref{eq14})$
yields (similar with proof of the Lemma 2) the differential system with respect 
to vector-function $x(t):$
$$\frac{dx_i}{dt}  = \sum_{k=0}^n c_{ik}x_k + o(||x||)$$
with conditions $x_i(0)=\Delta_i,\, i=1,\dots,n.$
Here 
$$c_{ik} = \left\{\begin{array}{l} {a_{ik}, i=1,\cdots,n, \, k=m+1,\dots, n}, \\ {a_{ik}+b_{ik} w_k^*, i=1,\cdots, n, \, k=1,\cdots, m.} \\ \end{array}\right.$$
Let us introduce the matrix $M=\{c_{ik}\}_{i,k=1}^n.$
Finally, we make\\

\noindent {\bf Remark.}
Let $\sum_{k=1}^n|\Delta_k|$ be sufficiently small.
If all the eigenvalues of the matrix $M$
has negative real parts, then exists solution to problem (\ref{eq14}), (\ref{eq15}),
this solution is unique and stabilizing to rest point $(0,0)$
as $t\rightarrow +\infty.$
Since polynomials $\sum_{s=0}^{N_k} m_{ks} w_k^s,$ $k=1,\dots,m$
can have several simple roots such as the corresponding matrix 
$M$ has only eigenvalues with negative real parts,
then solution to problem (\ref{eq14}), (\ref{eq15}) in general case may have
several stable solutions $x(t)$.
The simple roots of these polynomials, under which the matrix $ M $ has
an eigenvalue with positive real part will correspond to an unstable
 solution $ x (t). $

\begin{example}

$$\left\{\begin{array}{l} {\frac{dx}{dt} = \alpha x+\beta u +u^2 +x^3}, \\ {x(0)=\Delta,} \\
ax^2 + 2b x u +u^2 = 0, 0\leq t <\infty. \end{array}\right. $$
Here $(0,0)$ is the rest point. Under assumption $u=cx, $ where $\, c $ is $ const$ 
we get the following quadratic equation 
$c^2+2bc +a =0.$
Then $u$ is double-valued 
\beq
u_{1,2} = x(t)(-b\pm \sqrt{b^2-a}).
\eeq
Let $a<b^2.$

We substitute the found values of the function $ u $ into the differential equation.
Then the problem of determining the function $ x(t) $ is reduced to the solution of two
 Cauchy problems
$$\left\{\begin{array}{l} {\frac{dx_ {\pm}}{dt} = (\alpha +\beta(-b\pm \sqrt{b^2-a}) x_{\pm}  
+ (-b\pm \sqrt{b^2-a})^2 x_{\pm}^2 + x_{\pm}^3}, \\ {x_{\pm}(0) = \Delta.}\\ 
\end{array}\right. 
$$
Let $\alpha+\beta(-b - \sqrt{b^2-a})<0.$ Then exists branch
 $x_{-}(t)$  for small $|\Delta|$ for $t\geq 0$ and stabilizing to zero as $t\rightarrow +\infty.$

\end{example}

\section{Possible generalizations}

In this work until now only autonomus systems were considered. This  can be relaxed.
For example, if we have system
$$\left\{\begin{array}{l} { \opA \frac{dx}{dt} = (\opA_1 +\tilde{\opA}_1(t))x(t) + (\opA_2 +\tilde{\opA}_2(t)) u(t) +\opR(x,u,t)  }, \\ 0 = \opA_3 x +\opA_4 u + \opr(x,u,t),\\ 
\end{array}\right. 
$$
 where $\tilde{\opA}_1(t) \rightarrow 0,$ $\tilde{\opA}_2(t) \rightarrow 0$
as $t\rightarrow +\infty,$
$||\opR(x,u,t)||=o(||x||+||u||)$
and $||\opr(x,u,t)||=o(||x||+||u||)$
for $||x||+||u|| \rightarrow 0$  coverges uniformly $t\geq 0,$
 then results of Theorem 1 remains correct.

In the theory of systems (\ref {eq1}), (\ref {eq2}), the most difficult case was when
the Frechet derivative $ \frac {\partial} {\partial u} \opG(x, u) $ is not invertible at the rest point, and therefore the implicit operator theorem is not fulfilled for the map $ \opG(x, u) = 0$.

In sec. 5   only one case was adressed when this condition is not satisfied and
 solution is branching. Other more complex cases of branching solutions can be investigated using the results of modern analytical branching theory
solutions of non-linear equations obtained in the works of V.A. Trenogin, B.V. Loginov, N.A. Sidorov, A.D. Bruno,  M.G. Krein,  J.~Toland \cite{sid82, lit11, lit12, lit9, lit4, lit1, lit8} et al.
Equally interesting is the problem of analyzing systems (\ref {eq1}), (\ref {eq2}) with a discontinuity
in a neighborhood of the rest points, when the stability condition in the first approximation is not satisfied, and more advanced methods must be used, for example, methods related to the construction
Lyapunov functions, to evaluate the location of potential blow-up points using
method of convex majorants of L.V. Kantorovich used in works \cite{lit10, lit7,litF}.

In this case, when developing algorithms for analyzing stability and constructing estimates of the regions of attraction of the rest points of the power systems of input-output type, it is expedient to use methods based on the theory of the Lyapunov vector-function.

Finally, it is interesting to consider the system (\ref {eq1}), (\ref {eq2})
with rest points for an irreversible operator $ \opA. $
In this case, the standard Cauchy problem can has no classical solutions and it is advisable to introduce other initial conditions.
If the irreversible operator $ \opA $ admits a finite-length skeleton decomposition, then new correct initial conditions for the problem (\ref {eq1}), (\ref {eq2}) can be formulated using
the results of the works \cite{lit12,  lit14}.

\acknowledgements{\rm This work is 
fulfilled as part of the programm for Irkutsk State University development for  2015--2019 under the project ``Singular operator-differential systems of equations and mathematical models with parameters''. It is partly supported by the programm of international 
scientific collaboration of China and Russia under Grant No. 2015DFR70850, NSFC 
grant  No. 61673398 and programm of fundamental research of SB RAS, reg.~No.~ÀÀÀÀ-À17-117030310442-8, research project III.17.3.1. The results of this manuscript were partly reported on the Russian-Chinese Workshop "Mathematical Modeling of Renewable and Isolated Hybrid Power Systems", lake Baikal, 2–6 August 2017 \cite{isu1,isu2}.}


\begin{thebibliography}{99}


\bibitem{lit00} Ayasun S., Nwankpa C.O., Kwatny H.G. {Computation of singular and
singularity induced bifurcation points of differential-algebraic power system model.} IEEE Transactions on Circuits and Systems - I: Fundamental
Theory and Applications. (2004),  51(8):~1525--1538.
https://doi.org/10.1109/TCSI.2004.832741

\bibitem{bel}
Machowski J.,  Bialek J.W., Bumby J.R. { 
Power system dynamics. Stability and control}.  Oxford. John Wiley, 2008, 658~p.

\bibitem{lit0} Milano F. {Power system modelling and scripting,} Berlin, Springer, 2010,  578~p. https://doi.org/10.1007/978-3-642-13669-6

\bibitem{vorop} Voropai N.I., Kurbatsky V.G. et al.  {Complex of intelligent tools for preventing major accidents in electric power systems}.  Novsibirsk. 
Nauka, 2016, 332~p. (in Russian)

\bibitem{lit16}
Barbashin E.A. {Introduction to stability theory.} M. Libercom, 2014, 230~p. (in Russian)


\bibitem{joh}
 Sj\"{o}berg J., Fujimoto K.,  Glad T. {Model reduction of nonlinear differential-algebraic equations.} IFAC Proceedings Volumes. (2007), 40 (12): 176--181.
https://doi.org/10.3182/20070822-3-ZA-2920.00030

\bibitem{nonlin} Khalil H. K.  {Nonlinear systems,} Prentice hall, 1991.

\bibitem{lit7} Sidorov D., Sidorov N. {Convex majorants method in the theory of nonlinear Volterra
equations}.  Banach J. of Mathematical Analysis. (2014),  6(1): 1--10.
 https://doi.org/10.15352/bjma/1337014661


\bibitem{isu1}
Sidorov N.A., Sidorov D.N., Li Y., "Oblasti prityazheniya tochek ravnovesiya nelinejnyh sistem: ustojchivost', vetvlenie i razrushenie reshenij", IIGU Ser. Matematika. (2018), 23: 46-–63.
(in Russian)
https://doi.org/10.26516/1997-7670.2018.23.46


\bibitem{lit6} Erugin N.P. {The Book for Reading on General Course of Differential Equations.} Minsk: Nauka i Tekhnika. 1972, 668~p. (in Russian)


\bibitem{lit3} Trenogin V.A. {Functional analysis. }Moscow, Fizmatlit, 2002, 488~p. (in Russian)



\bibitem{lit10} Sidorov D.N. {Existence and blow-up of Kantorovich principal continuous solutions of nonlinear integral equations,} Differential Equations. (2014), 50(9): 1217--1224.
https://doi.org/10.1134/S0012266114090080

\bibitem{sid82} Sidorov N.A. {General issues of regularization in branching problems.}
Irkutsk. ISU Publ., 1982, 312~p. (in Russian)




\bibitem{lit11} Buffoni B., Toland J.
{Analytic Theory of Global Bifurcation: An Introduction.}
Princeton series in applied mathematics,
Princeton University Press, 2003. 169~p.
https://doi.org/10.1515/9781400884339










\bibitem{lit12} Sidorov N., Loginov B., Sinitsyn A., Falaleev M.
{Lyapunov-Schmidt methods in nonlinear analysis and applications.} { Springer Series: Mathematics and Its Applications, Vol. 550,}  2013, 568~p.
 https://doi.org/10.1007/978-94-017-2122-6



\bibitem{lit9} 
Vainberg M. M., Trenogin V. A. {Theory of branching of solutions of non-linear equations}. Leyden, 1974.



\bibitem{lit4}  Daleckii Ju. L.,  Krein M.G. {
Stability of solutions of differential equations in Banach space.} Ser. ``Translations of Mathematical Monographs'' Vol. 43. Rhode Island, AMS Publ. 2002. 386 ~p.



\bibitem{lit1} Demidovich B. P. { Lectures on mathematical stability theory,}
Moscow, Nauka, 1967, 471~p. (in Russian)



%
%


\bibitem{lit8} Sidorov N.A., Trenogin V.A. {Bifurcation points of 
nonlinear equation.} In the book ``Nonlinear analysis and nonlinear differential equations''.
Edts V.A. Trenogin and A.F. Filippov. Moscow. Fizmatlit. 2013. P.~5--50. (in Russian)





\bibitem{litF} Sidorov D. {Integral Dynamical Models:
Singularities, Signals and Control}; Ed. by L. O. Chua, Singapore, London: World Scientific Publ.,
2015, vol. 87 of {\it World Scientific Series on Nonlinear Science, Series A}, 258~p.
https://doi.org/10.1142/9789814619196\_bmatter



\bibitem{lit14}
Sidorov D.N., Sidorov N.A.  {Solution of irregular systems of partial differential equations using skeleton decomposition of linear operators}.  Vestn. YuUrGU. Ser. Matem. modelirovanie i programmirovanie. (2017), 10(2): 63--73.  https://doi.org/10.14529/mmp170205




\bibitem{isu2}
Sidorov D.N., Li Y. Rossijsko-kitajskij seminar ``Matematicheskoe modelirovanie ehlektroehnergeticheskih sistem na vozobnovlyaemyh istochnikah ehnergii i izolirovannye gibridnye sistemy ehlektrosnabzheniya'', IIGU Ser. Matematika. (2017), 21: 122--126.
(in Russian)
https://doi.org/10.26516/1997-7670.2017.21.122

\end{thebibliography}
\end{document}